\documentclass[10pt]{article}
\usepackage{amsthm,amssymb,amsmath,amsfonts}
\usepackage{graphicx} 
\usepackage{enumitem}
\usepackage{graphics,mathrsfs,mathtools,soul,csquotes}
\usepackage[color=red!40]{todonotes}
\usepackage{enumitem}
\usepackage{bm,dsfont}
\usepackage{tikz-cd}
\usetikzlibrary{positioning,arrows,scopes}
\usepackage{fullpage}

\definecolor{cadmiumgreen}{rgb}{0.0, 0.42, 0.24}
\usepackage[
colorlinks,citecolor=cadmiumgreen,
pagebackref,
pdfauthor={Aleyah Dawkins,
Vishal Gupta, Mark Kempton, William Linz, 
Jeremy Quail, Harry Richman, Zachary Stier}, 
pdfstartview ={FitV},
]{hyperref}
\hypersetup{
pdftitle={A Ricci flow on graphs from effective resistance},
}

\newtheorem{thm}{Theorem}[section]
\newtheorem*{thm*}{Theorem}

\newtheorem{prop}[thm]{Proposition}
\newtheorem{lem}[thm]{Lemma}

\theoremstyle{definition}
\newtheorem{prob}[thm]{Problem}
\newtheorem{dfn}[thm]{Definition}
\newtheorem*{dfn*}{Definition}
\newtheorem{eg}[thm]{Example}
\newtheorem{rmk}[thm]{Remark}

\numberwithin{equation}{section}

\newcommand{\RR}{\mathbb{R}}
\newcommand{\cN}{\mathcal{N}}

\newcommand{\ellzero}{\ell^{(0)}}

\newcommand{\tderiv}[1]{\frac{d\,{#1}}{dt}}
\newcommand{\partialderiv}[2]{\frac{\partial\,{#1}}{\partial\,{#2}}}

\newcommand{\curv}{\mathrm{K}} 

\usepackage{authblk}
\title{A Ricci flow on graphs from effective resistance}
\author{Aleyah Dawkins}
\affil{George Mason University}
\author{Vishal Gupta}
\affil{University of Delaware}
\author{Mark Kempton}
\affil{Brigham Young University}
\author{William Linz}
\affil{University of South Carolina}
\author{Jeremy Quail}
\affil{University of Vermont}
\author{Harry Richman}
\affil{Fred Hutchinson Cancer Center}
\author{Zachary Stier}
\affil{University of California, Berkeley}
\date{}

\begin{document}

\maketitle

\begin{abstract}
    In this paper, we introduce a new notion of curvature on the edges of a graph that is defined in terms of effective resistances.
    We call this the Ricci--Foster curvature. 
    We study the Ricci flow resulting from this curvature.
    We prove the existence of solutions to Ricci flow on short time intervals, and prove that Ricci flow preserves graphs with nonnegative (resp. positive) curvature.
\end{abstract}


\section{Introduction}

Recent years have seen considerable work in taking ideas from differential geometry and applying them to discrete settings on graphs.  Of particular significance in this realm has been defining notions of curvature for graphs. The notions of Ricci curvature and scalar curvature are critical tools in differential geometry for understanding manifolds.  Of particular note is the notion of Ricci flow associated with curvature, which was a principal tool in the proof of the Poincar\'e conjecture~\cite{perelman}.  
Various notions of curvature have been defined for graphs, including the so-called Bakry--\'Emery curvature~\cite{BE}, Ollivier curvature~\cite{ollivier}, Lin--Lu--Yau curvature~\cite{LLY}, Forman curvature~\cite{forman}, and combinatorial curvature~\cite{higuchi}.  For some of these, particularly for the Ollivier and Forman curvatures, an associated Ricci flow has been studied~\cite{NLLG,BLLWY,CKLMPS}.

Recently, Devriendt and Lambiotte~\cite{DL} introduced a notion of Ricci curvature for graphs, based on effective resistance (defined in Section \ref{sec:res} below). 
They defined a curvature for the vertices and edges of a graph, and studied some properties of Ricci flow resulting from this notion of edge curvature.

The main contribution of this work is to introduce a variant on the edge curvature of Devriendt and Lambiotte \cite{DL} which yields the same vertex curvature (which we view as a {\em scalar curvature}) of that paper, but which leads to a more natural notion of Ricci flow.

Classically, the Ricci curvature and Ricci flow satisfy several key properties, including:

\begin{enumerate}
    \item The Ricci tensor is invariant under uniform scaling of the metric~\cite{topping}.

    \item Total volume is weakly decreasing under Ricci flow on a closed manifold with nonnegative scalar curvature~\cite[Corollary 3.2.6]{topping}.

    \item Nonnegative (or positive) scalar curvature is preserved under Ricci flow on a closed manifold~\cite[Corollary 3.2.3]{topping}.

    \item Ricci flow has no periodic orbits, on the space of Riemannian metrics modulo diffeomorphism and scaling~\cite{perelman}.
\end{enumerate}
The Ricci flow we introduce on graphs satisfies the first three of these properties;
see Proposition~\ref{prop:curv-rescaling}, Proposition~\ref{prop:total-length-flow}, and Theorem~\ref{thm:positively-curved}. 
The analogue of the fourth property we leave open, but conjecture that it also holds.

In this paper, we accomplish the following:
\begin{itemize}
\item We introduce a new notion of resistance-based Ricci curvature on the edges of a graph, which we call the {\em Ricci--Foster curvature}.

\item We show the existence of Ricci flow for this curvature on an arbitrary edge-weighted graph, for a short time interval.

\item We show that under this notion of Ricci curvature, Ricci flow preserves nonnegatively curved graphs.
\end{itemize}

On a graph $G = (V,E)$, for each edge $e = uv$ let $\ell_{e}$ denote the {\em resistance} of $e$, which is a positive real number.
For any vertices $u, v$ in $G$, let $\omega_{uv}$ denote the {\em effective resistance} between $u$ and $v$.

\begin{dfn}[Ricci--Foster curvature]
    If $e = uv$ is an edge in a weighted graph, let $\curv_e$ denote the following {\em Ricci--Foster curvature} of $e$:
    \[\curv_{e} = \frac1{\deg_u} + \frac1{\deg_v} - \frac{\omega_{uv}}{\ell_{e}}.\]
\end{dfn}

\begin{thm}\label{thm:flow-existence}
    Suppose $(G, \ellzero)$ is an edge-weighted graph.
    Consider the Ricci--Foster flow on $\{\ell_e(t) : e \in E(G)\}$ defined by the system of differential equations
    \begin{equation}\label{eq:ricci-flow-intro}
        \tderiv{} \ell_e(t) = - \curv_e(\ell(t))
        \qquad\text{and}\qquad
        \ell_e(0) = \ellzero_{e}
    \end{equation}
    for each edge $e \in E(G)$.
    Then for some $T > 0$ there exists a unique solution to this system of differential equations on the interval $t \in [0, T)$.
\end{thm}

Say an edge-weighted graph $(G, \ell)$ is {\em nonnegatively curved} (resp.\ {\em positively curved}) if every edge of $G$ has nonnegative (resp.\ positive) Ricci--Foster curvature.

\begin{thm}\label{thm:positively-curved}
    Suppose $(G, \ellzero)$ is a finite, edge-weighted graph,
    and let $(G, \ell(t))$ for $t \in [0, T)$ denote the family of graphs obtained from $(G, \ellzero)$ by applying Ricci--Foster flow~\eqref{eq:ricci-flow-intro}.
    If $(G, \ellzero)$ is nonnegatively curved, then so is $(G, \ell(t))$ for all $t \in [0, T)$.
\end{thm}

\subsection*{Acknowledgements}
This work started at the 2023 American Mathematical Society Mathematical Research Communities on Ricci Curvatures of Graphs and Applications to Data Science, which was supported by the National Science Foundation under Grant Number DMS 1916439. 
We thank Fan Chung, Mark Kempton, Wuchen Li, Linyuan Lu, and Zhiyu Wang for organizing this workshop.
Our investigation was greatly assisted by the online graph curvature calculator~\cite{CKLLS-calculator}.
We thank Karel Devriendt, Andrea Ottolini, and Stefan Steinerberger for helpful conversations.
WL was also partially supported by NSF RTG Grant DMS 2038080. ZS was additionally supported by NSF grant DGE 2146752.

\section{Effective resistance}\label{sec:res}

In this section, we recall the definition and some properties of effective resistance.
We consider a graph $G = (V, E)$ where each edge $e$ has a resistance $\ell_e$.
The resistance is always a positive real number.
We will sometimes also refer to $\ell_e$ as the {\em length} of edge $e$.

Let $\omega_{xy}$ denote the {\em effective resistance} between vertices $x$ and $y$.  There are many equivalent ways to define the effective resistance in a graph.  We make use of the following.
\begin{dfn}[Effective resistance via spanning trees {\cite[Exercise 1.1]{grimmett}}]\label{dfn:eff-res}
    The effective resistance between vertices $x$ and $y$ satisfies
    \begin{equation}\label{eq:tree-ratio}
        \omega_{xy} = \frac{\tau(G/x y; \ell)}{\tau(G; \ell)}
    \end{equation}
    where $G / xy$ denotes the graph obtained from $G$ by identifying $x$ and $y$,
    and $\tau(G)$ is the weighted sum over all spanning trees
    \[
        \tau(G; \ell) = \sum_{T} \prod_{e \not \in T} \ell_e .
    \]
When $uv = e$ is an edge, we sometimes use $\omega_{e}$ to mean $\omega_{uv}$.
\end{dfn}

If $uv$ is itself an edge, then the effective resistance $\omega_{uv}$ takes into account the resistances of {\em all} edges on all paths connecting $u$ and $v$, rather than just the edge $uv$ itself.
Recall that $\ell_{e}$ denotes the resistance of an edge $e$.
In general, if $e = uv$ is an edge, we have
\[
    \omega_{uv} \leq \ell_{e}.
\]
To be more precise, if $uv = e$ forms an edge, then by the rule for combining resistances in parallel~\cite[Chapter 1.3]{grimmett} we have
\[
    \omega_{uv} = \omega_{uv}(G) = \frac{\ell_{e} \, \omega_{uv}(G \setminus e)}{\ell_{e} + \omega_{uv}(G \setminus e)},
\]
where $G \setminus e$ denotes the edge-deleted graph.  In this formula, if $e$ is a bridge (an edge whose removal disconnects the graph) then $\omega_{uv}(G\setminus e)$ is infinite and the expression above is indeterminate.  We take $\omega_{uv}=\ell_e$ in this case, as we see from the following rewriting of this formula.
\begin{equation}\label{eq:resistance-ratio}
    \frac{\omega_{uv}}{\ell_{e}} = \frac{\omega_{uv}(G \setminus e)}{\ell_{e} + \omega_{uv}(G \setminus e)}
    = 1 - \frac{\ell_{e}}{\ell_{e} + \omega_{uv}(G \setminus e)}.
\end{equation}

\begin{prop}[Rayleigh's monotonicity law]\label{prop:rayleigh}
    For any edge $e$ and any vertices $x$ and $y$, we have
    $\displaystyle {\partialderiv{\omega_{xy}}{\ell_e} \geq 0}$.
\end{prop}
For a proof, see Grimmet~\cite[Theorem 1.29]{grimmett} or Doyle and Snell~\cite[Chapter 1.4]{DS}.

\begin{rmk}
    In fact, we can be more precise concerning the value of $\displaystyle \partialderiv{\omega_{xy}}{\ell_e}$.
    Suppose we let $i_e^{xy}$ denote the amount of current flowing through edge $e$, when one unit of current is sent from $x$ to $y$.
    Then according to \cite[Exercise 1.4.7]{DS}, we have
    $\displaystyle
        \partialderiv{\omega_{xy}}{\ell_e} = \left(i_e^{xy}\right)^2.
    $
\end{rmk}

\section{Edge curvature}

In this section, we recall the definition of curvature based on effective resistance from the introduction, and explore some of its properties.
\begin{dfn*}[Ricci--Foster curvature]
    Suppose $(G, \ell)$ is an edge-weighted graph where $\ell_e$ denotes the resistance of $e$ for each edge in $E(G)$.
    Consider the {edge curvature} defined as follows:
    on each edge $e = uv$ let
    \begin{equation}\label{eq:curvature}
        \curv_{e} = \frac{1}{\deg_u} + \frac{1}{\deg_v} - \frac{\omega_{uv}}{\ell_{e}}
    \end{equation}
    where $\omega_{uv}$ denotes the effective resistance between vertices $u$ and $v$.
    We call $\curv_e$ the {\em Ricci--Foster curvature} of $e$.
\end{dfn*}

\begin{rmk}
    Devriendt and Lambiotte~\cite{DL} define the scalar curvature $p_u$ at a node $u$ as 
    \[
        p_u = 1 - \frac12\sum_{v : v\sim u} \frac{\omega_{uv}}{\ell_{uv}}.
    \]
    One could view our definition of Ricci--Foster curvature as arising from thinking of an edge $e=uv$ as consisting of two directed arcs $\overrightarrow{uv}$ and $\overleftarrow{uv}$ ($ = \overrightarrow{vu}$), and then assigning each directed arc the curvature 
    \[
        K_{\overrightarrow{uv}} = \frac{1}{\deg_u} - \frac12\frac{\omega_{uv}}{\ell_e},
    \] 
    from which we see that $K_e = K_{\overrightarrow{uv}}+K_{\overleftarrow{uv}}$.  
    Then we see that the scalar curvature at a node can be expressed as a sum of Ricci curvatures on the outgoing edges:
    \[
    p_u = \sum_{v : v \sim u}K_{\overrightarrow{uv}}.
    \]
\end{rmk}

\begin{rmk}
    The scalar curvature $p_u$ of \cite{DL} is further studied in \cite{DOS}, with a different normalizing factor, where the curvature is shown to satisfy nice properties analogous to the situation on manifolds.
\end{rmk}

\subsection{Examples}

\begin{eg}[Trees]
    If $G$ is a (weighted) tree, then $\omega_{uv} = \ell_{e}$ for every edge $uv = e$.
    Thus the edge curvature is
    \[
        \curv_{uv} = \frac1{\deg_u} + \frac1{\deg_v} - 1.
    \]
    An edge has positive curvature if $\deg_u = 1$ or $\deg_v = 1$,
    and has zero curvature if $\deg_u = \deg_v = 2$.
    Otherwise, the edge curvature is negative.

    For an explicit example, the tree in Figure~\ref{fig:tree-curv} has its edge curvatures indicated.
    We highlight positively-curved edges in blue, and negatively-curved edges in red.
    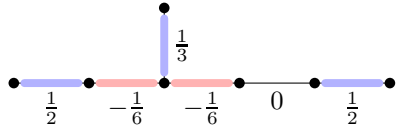
\begin{figure}[ht]
    \centering
    \begin{tikzpicture}
        \coordinate (A) at (0,0);
        \coordinate (B) at (1,0);
        \coordinate (C) at (2,0);
        \coordinate (D) at (3,0);
        \coordinate (E) at (4,0);
        \coordinate (F) at (5,0);
        \coordinate (G) at (2,1);
        
        \draw (A) --node[below] {$\frac12$} (B) --node[below] {$-\frac16$} (C) 
            --node[below] {$-\frac16$} (D) --node[below] {$0$} (E) --node[below] {$\frac12$} (F);
        \draw (C) --node[right] {$\frac13$} (G);
        \foreach \p in {A, B, C, D, E, F, G} {
            \fill (\p) circle (2pt);
        }
		\begin{scope}[color=blue!30,opacity=0.5,line width=3pt,line cap=round,shorten <=4pt, shorten >=4pt]
			\draw (A) -- (B);
			\draw (C) -- (G);
			\draw (E) -- (F);
		\end{scope}
		\begin{scope}[color=red!30,opacity=0.5,line width=3pt,line cap=round,shorten <=4pt, shorten >=4pt]
			\draw (B) -- (C);
			\draw (C) -- (D);
		\end{scope}
    \end{tikzpicture}
    \caption{A tree with its Ricci--Foster edge curvatures.}
    \label{fig:tree-curv}
    \end{figure}
\end{eg}

\begin{eg}[Infinite regular tree]
\label{eg:infinite-tree}
    Suppose $G$ is the infinite regular tree where each vertex has degree $d$.
    Then for each edge,
    \[
        \curv_e = \frac{2}{d} - 1 = \frac{2 - d}{d}.
    \]
    If $d \geq 3$ then each edge has negative curvature.
\end{eg}

\begin{eg}[Cycles]
    If $G$ is a cycle on $n$ vertices and edges, then $\curv_e = \frac{1}{n}$ for each $n$.
    We omit the calculation, which is simplified by applying Proposition~\ref{prop:curvature-sum}.
    
    More generally, if $G$ is a cycle with edge resistances $\{\ell_1, \ldots, \ell_n\}$, then 
    the curvature of the $i$-th edge is
    \[
        \curv_i = \frac12 + \frac12 - \frac{\omega_i}{\ell_i} = \frac{\ell_i}{\sum_{j \in E} \ell_{j}}.
    \]
    In other words, the curvature of each edge is proportional to its length.
\end{eg}

\begin{eg}[Barbell graph]\label{eg:barbell}
    Suppose $G$ is the graph shown in Figure~\ref{fig:barbell}.
    The single bridge edge has negative curvature, while the other edges have positive curvature.
    \begin{figure}[ht]
    \centering
    \raisebox{-0.5\height}{\begin{tikzpicture}
    	\coordinate (1) at (-1.2,0.8);
	  \coordinate (2) at (-1.2,-0.8);
    	\coordinate (3) at (0,0);
	    \coordinate (4) at (2,0);
    	\coordinate (5) at (3.2,-0.8);
        \coordinate (6) at (3.2,0.8);
	
	  \foreach \c in {1,2,3,4,5,6} {
	   	\filldraw[black] (\c) circle (2pt);
    	}
    
        \draw (3) --node[above right] {$\frac{1}{6}$} (1) --node[left] {$\frac{1}{3}$} (2) --node[below right] {$\frac{1}{6}$} cycle;
    	\draw (3) --node[above] {$-\frac{1}{3}$} (4);
        \draw (4) --node[below left] {$\frac{1}{6}$} (5) --node[right] {$\frac{1}{3}$} (6) --node[above left] {$\frac{1}{6}$} cycle;
		\begin{scope}[color=blue!30,opacity=0.5,line width=3pt,line cap=round,shorten <=4pt, shorten >=4pt]
			\draw (1) -- (2);
			\draw (1) -- (3);
			\draw (2) -- (3);
			\draw (4) -- (5);
			\draw (4) -- (6);
			\draw (5) -- (6);
		\end{scope}
		\begin{scope}[color=red!30,opacity=0.5,line width=3pt,line cap=round,shorten <=4pt, shorten >=4pt]
			\draw (3) -- (4);
        \end{scope}
    \end{tikzpicture}}
    \caption{Barbell graph and its edge curvatures.}
    \label{fig:barbell}
    \end{figure}
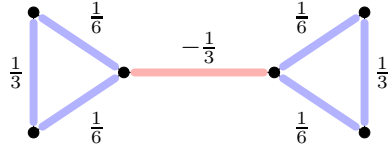
\end{eg}

\begin{eg}[House graph]
    Suppose $G$ is the graph shown in Figure~\ref{fig:house}.
    The edge curvatures are indicated in the figure next to each edge.
    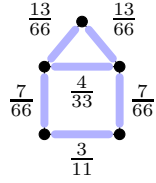
\begin{figure}[ht]
    \centering
    \raisebox{-0.5\height}{\begin{tikzpicture}
    	\coordinate (1) at (0,0);
	  \coordinate (2) at (1,0);
    	\coordinate (3) at (0,0.9);
	    \coordinate (4) at (1,0.9);
    	\coordinate (5) at (0.5,1.5);
	
	   \foreach \c in {1,2,3,4,5} {
	   	\filldraw[black] (\c) circle (2pt);
    	}
    
        \draw (1) --node[below] {$\frac{3}{11}$} (2) --node[right] {$\frac{7}{66}$} (4) 
            --node[below] {$\frac{4}{33}$} (3) --node[left] {$\frac{7}{66}$} cycle;
    	\draw (3) --node[above left] {$\frac{13}{66}$} (5) --node[above right] {$\frac{13}{66}$} (4);
		\begin{scope}[color=blue!30,opacity=0.5,line width=3pt,line cap=round,shorten <=4pt, shorten >=4pt]
			\draw (1) -- (2);
			\draw (1) -- (3);
			\draw (2) -- (4);
			\draw (3) -- (4);
			\draw (3) -- (5);
			\draw (4) -- (5);
		\end{scope}
    \end{tikzpicture}}
	
    
    \caption{House graph and its edge curvatures.}
    \label{fig:house}
    \end{figure}
\end{eg}

\subsection{Properties}

In this section, we note some basic properties of the Ricci--Foster curvature.

\begin{prop}\label{prop:curvature-bounds}
    For any edge $e$, the curvature satisfies $|\curv_e| \leq 1$.
\end{prop}
\begin{proof}
    For a lower bound on curvature, we have
    \[
        \curv_e = \frac{1}{\deg_u} + \frac{1}{\deg_v} - \frac{\omega_{uv}}{\ell_{e}} \geq - \frac{\omega_{uv}}{\ell_{e}} \geq -1.
    \]
    The last inequality follows from \eqref{eq:resistance-ratio}.
    
    As an upper bound for curvature, we consider two cases.
    First suppose that some endpoint of $e$ has degree one;
    we may assume by symmetry that $\deg_u = 1$.
    In this case, $e$ is a bridge edge, which means that $\frac{\omega_{uv}}{\ell_e} = 1$.
    Thus
    \[
        \curv_e = \frac{1}{\deg_u} + \frac{1}{\deg_v} - \frac{\omega_{uv}}{\ell_{e}} = \frac{1}{\deg_v} \leq 1.
    \]
    In the second case, suppose both $\deg_u \geq 2$ and $\deg_v \geq 2$.
    Then
    \[
        \curv_e = \frac{1}{\deg_u} + \frac{1}{\deg_v} - \frac{\omega_{uv}}{\ell_{e}} \leq \frac{1}{\deg_u} + \frac{1}{\deg_v} \leq 1
    \]
    as desired.
    Note that the resistance $\ell_e$ and effective resistance $\omega_{uv}$ are always positive (assuming $u \neq v$).
\end{proof}

\begin{prop}
\label{prop:curvature-sum}
    For any finite weighted graph, the sum of Ricci--Foster curvatures over all edges is 1.
\end{prop}
\begin{proof}
    We have
    \begin{align*}
        \sum_{e \in E} \curv_e &= \sum_{e \in E} \left(\frac1{\deg(e^+)} + \frac1{\deg(e^-)} - \frac{\omega_e}{\ell_e}\right) \\
        &= \sum_{e \in E} \left(\frac1{\deg(e^+)} + \frac1{\deg(e^-)}\right) - \sum_{e \in E} \frac{\omega_e}{\ell_e} \\
        &= \sum_{v \in V} \left(\sum_{e \in \cN(v)} \frac{1}{\deg_v}\right) - \sum_{e \in E} \frac{\omega_e}{\ell_e} 
        = |V| - \sum_{e \in E} \frac{\omega_e}{\ell_e}.
    \end{align*}
    In the last line, we use $\cN(v)$ to denote the set of edges incident to $v$.
    Thus it remains to show that
    \[
        \sum_{e \in E} \frac{\omega_e}{\ell_e} = |V| - 1.
    \]
    This is a result of Foster~\cite{foster}.
\end{proof}

The previous result explains our choice of naming the quantity \eqref{eq:curvature} the ``Ricci--Foster'' curvature.
For the rest of the paper, we will refer to the Ricci--Foster curvature simply as the curvature of an edge in a graph.
Note that for Proposition~\ref{prop:curvature-sum} to hold, it is necessary to assume that the graph is finite;
for an infinite-graph counterexample, see Example~\ref{eg:infinite-tree}.

\begin{prop}\label{prop:curv-rescaling}
    The curvature $\curv_e$ does not change under uniform rescaling of all edge resistances ${ \{\ell_f : f \in E\} }$.
\end{prop}
\begin{proof}
    From~\eqref{eq:curvature} it suffices to show that $\omega_{uv} / \ell_e$ is invariant under uniform rescaling.
    The effective resistance $\omega_{uv} = \omega_{uv}(\ell_f : f \in E)$ is a homogeneous degree-one function of the edge resistances (e.g., doubling the resistance of every edge will double $\omega_{uv}$).
    This property of effective resistance follows from Definition~\ref{dfn:eff-res}.
    Therefore the ratio $\omega_{uv} / \ell_e$ is invariant under uniform rescaling.
\end{proof}

We next consider how the curvature of an edge changes as a result of changing the length of a single (possibly different) edge.
In words, we find that the resistance of any edge is positively correlated with its curvature,
and negatively correlated with the curvature of any other edge.

\begin{prop}\label{prop:curvature-partials}
    \hfill
    \begin{enumerate}[label=(\alph*)]
        \item For a fixed edge $e$, we have
        $\displaystyle
            \partialderiv{}{\ell_e} \curv_e \geq 0.
        $
        \item For distinct edges $e \neq f$, we have
        $\displaystyle
            \partialderiv{}{\ell_f} \curv_e \leq 0.
        $
    \end{enumerate}
\end{prop}
\begin{proof}
    We first consider two special cases, when $e$ is either a loop edge or a bridge edge.
    If $e$ is a loop edge, we have $\frac{\omega_e}{\ell_e} = 0$ since the numerator vanishes, hence $\curv_e = \frac1{\deg(e^+)} + \frac1{\deg(e^-)}$ is constant.
    If $e$ is a bridge edge, we have $\frac{\omega_e}{\ell_e} = 1$ so again $\curv_e = \frac1{\deg(e^+)} + \frac1{\deg(e^-)} - 1$ is constant. 
    In both cases,  $\partialderiv{}{\ell_f} \curv_e = 0$ for $f = e$ and $f \neq e$, which satisfies parts (a) and (b).
    
    For the rest of the proof, suppose $e$ is not a loop or bridge edge, so that we can use the expressions \eqref{eq:resistance-ratio} for the effective resistance  ratio $\omega_e / \ell_e$.
    For part (a),
    \begin{align*}
        \partialderiv{}{\ell_e} \curv_e &= \partialderiv{}{\ell_e} \left(\frac{1}{\deg(e^+)} + \frac{1}{\deg(e^-)} - \frac{\omega_e(G \setminus e)}{\ell_e + \omega_e(G \setminus e)}\right) \\
        &= \frac{\omega_e(G \setminus e)}{(\ell_e + \omega_e(G \setminus e))^2} \geq 0.
    \end{align*}
    Here we use the fact that $\partialderiv{\omega_e(G \setminus e)}{\ell_e} = 0$, since the edge $e$ does not appear in $G \setminus e$.
    
    For part (b), since $e$ and $f$ are assumed distinct we have $\partialderiv{\ell_e}{\ell_f} = 0$.
    Thus
    \begin{align*}
        \partialderiv{}{\ell_f} \curv_e
        &= \partialderiv{}{\ell_f} \left(\frac{1}{\deg(e^+)} + \frac{1}{\deg(e^-)} - 1 + \frac{\ell_e}{\ell_e + \omega_e(G \setminus e)}\right) \\
        &= \partialderiv{}{\ell_f} \left( \frac{\ell_e}{\ell_e + \omega_e(G \setminus e)}\right) \\
        &= - \frac{\ell_e}{(\ell_e + \omega_e(G \setminus e))^2} \cdot \partialderiv{\omega_e(G \setminus e)}{\ell_f}  \leq 0.
    \end{align*}
    The last inequality applies Rayleigh's law (Proposition~\ref{prop:rayleigh}), 
    $\partialderiv{\omega_e(G \setminus e)}{\ell_f} \geq 0$.
\end{proof}

\subsection{Edge refinement}

In this section, we show that the Ricci--Foster curvature is nicely compatible with a notion of edge refinement on edge-weighted graphs.
If $(G, \ell)$ is an edge-weighted graph, we say $(H, \ell')$ is obtained by {\em subdivision} of an edge $e = uv$ of $(G, \ell)$ if $H$ is obtained by adding two new edges $e_1=ux$ and $e_2=vx$ to $G\setminus e$ and assigning lengths $\ell_{e_1}$ and $\ell_{e_2}$ such that $\ell_e = \ell_{e_1} + \ell_{e_2}$. 

An equivalence class of edge-weighted graphs under repeated edge refinement is often referred to as a ``metric graph'' or ``metrized graph'' or ``quantum graph''~\cite{friedlander,BR,BK}.

\begin{dfn}[Foster coefficient]
    Given a weighted graph $(G, \ell)$, the {\em Foster coefficient} of an edge $e \in E(G)$ is
    \[
        F(e) = \frac{\ell_e}{\ell_e + \omega_e(G \setminus e)}.
    \]
\end{dfn}
In terms of Foster coefficients, the curvature satisfies
\begin{equation}\label{eq:curv-foster}
    \curv_e = \frac{1}{\deg(e^+)} + \frac{1}{\deg(e^-)} - 1 + F(e).
\end{equation}
Foster's work~\cite{foster} implies that $\sum_e F(e) = |E| - |V| + 1$.

\begin{lem}\label{lem:foster-subdiv}
    Suppose $G$ contains an edge $e$, and $H$ is obtained from $G$ by subdividing $e$ into two edges $e_1 \cup e_2$.
    Then $F^G(e) = F^H(e_1) + F^H(e_2)$.
\end{lem}
\begin{proof}
    Suppose $e = uv$, and the subdivided graph $H$ is shown in Figure~\ref{fig:edge-subdiv}, with $x$ the new vertex after subdivision.
    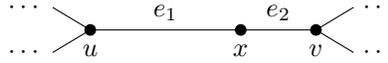
\begin{figure}[ht]
    \centering
    \begin{tikzpicture}
        \coordinate (A) at (0,0);
        \coordinate (B) at (2,0);
        \coordinate (C) at (3,0);

        \draw (A) --node[above]{$e_1$} (B) --node[above]{$e_2$} (C);
        \draw (A) -- +(-0.5,0.3) node[left] {$\cdots$};
        \draw (A) -- +(-0.5,-0.3) node[left] {$\cdots$};
        \draw (C) -- +(0.5,0.3) node[right] {$\cdots$};
        \draw (C) -- +(0.5,-0.3) node[right] {$\cdots$};
        \foreach \c in {A,B,C} {
            \filldraw (\c) circle (2pt);
        }
        \node[below=0.2em] at (A) {$u$};
        \node[below=0.2em] at (B) {$x$};
        \node[below=0.2em] at (C) {$v$};
    \end{tikzpicture}
    \caption{Edge after subdivision.}
    \label{fig:edge-subdiv}
    \end{figure}

    We have 
    \[
    \omega_{ux}(H \setminus e_1) = \ell_{e_2} + \omega_{uv}(G \setminus e)
    \qquad\text{and}\qquad
    \omega_{xv}(H \setminus e_2) = \ell_{e_1} + \omega_{uv}(G \setminus e).
    \]
    Thus
    \begin{align*}
        F^H(e_1) = \frac{\ell_{e_1}}{\ell_{e_1} + \omega_{ux}(H \setminus e_1)}
        &= \frac{\ell_{e_1}}{\ell_{e_1} + \ell_{e_2} + \omega_{uv}(G \setminus e)}
        = \frac{\ell_{e_1}}{\ell_e + \omega_{uv}(G \setminus e)},
    \end{align*}
    and similarly $\displaystyle F^H(e_2) = \frac{\ell_{e_2}}{\ell_e + \omega_{uv}(G \setminus e)}$.
    The result follows, since $\ell_e = \ell_{e_1} + \ell_{e_2}$.
\end{proof}

\begin{rmk}
    The previous lemma also follows from the fact that on the underlying metric graph of $(G, \ell)$, there exists an absolutely continuous measure $\mu$ such that the Foster coefficient $F(e)$ is equal to $\mu(e)$.
    (The measure $\mu(A)$ is defined for all Borel sets of the metric graph.)
    The measure $\mu$ is the continuous part of the {\em canonical measure} of $(G, \ell)$~\cite{BR}, a construction which was motivated by arithmetic geometry~\cite{CR}.    
\end{rmk}

\begin{thm}\label{thm:curv-subdiv}
    Suppose $G$ contains an edge $e$, and $H$ is obtained from $G$ by subdividing $e$ into two edges $e_1 \cup e_2$.
    Then $\curv^G_e = \curv^H_{e_1} + \curv^H_{e_2}$.
\end{thm}
\begin{proof}
    Suppose $e = uv$, and the subdivided graph $H$ is shown in Figure~\ref{fig:edge-subdiv}, with $x$ the new vertex after subdivision.
    
    By~\eqref{eq:curv-foster}, we have
    \begin{align*}
        \curv^H_{e_1} + \curv^H_{e_2} &= \left(\frac{1}{\deg_u} + \frac{1}{2} - 1 + F^H(e_1)\right) + \left(\frac{1}{2} + \frac{1}{\deg_v} - 1 + F^H(e_2)\right) \\
        &= \frac{1}{\deg_u} + \frac{1}{\deg_v} - 1 + F^H(e_1) + F^H(e_2),
    \end{align*}
    so it suffices to recall that $F^H(e_1) + F^H(e_2) = F^G(e)$ (Lemma~\ref{lem:foster-subdiv}).
    
\end{proof}

\begin{eg}[Subdividing edges and curvature]
    Consider the graph $(G, \ell)$ shown in Figure~\ref{fig:graph-subdiv}, with edge resistances indicated on the left.
    The edge curvatures are shown on the right.
    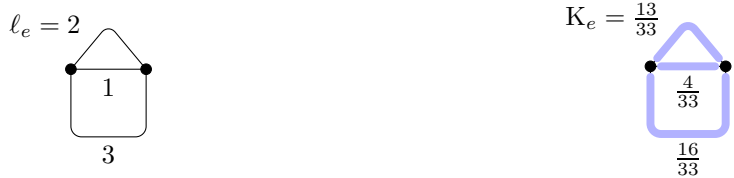
\begin{figure}[ht]
    \begin{minipage}{0.45\textwidth}
        \centering
        \begin{tikzpicture}
	  \coordinate (1) at (0,0);
	  \coordinate (2) at (1,0);
	  \coordinate (3) at (0,0.9);
	  \coordinate (4) at (1,0.9);
	  \coordinate (5) at (0.5,1.5);
	
        \foreach \c in {3,4} {
	  \filldraw[black] (\c) circle (2pt);
    	}
        
    	\draw[rounded corners] (3) -- (1) --node[below] {$3$} (2) -- (4);
        \draw[rounded corners] (4) --node[below] {$1$} (3);
    	\draw[rounded corners] (3) --node[above left] {$\ell_e = 2$} (5) -- (4);
        \end{tikzpicture}
    \end{minipage}
    \begin{minipage}{0.45\textwidth}
        \centering
        \begin{tikzpicture}
	  \coordinate (1) at (0,0);
	  \coordinate (2) at (1,0);
	  \coordinate (3) at (0,0.9);
	  \coordinate (4) at (1,0.9);
	  \coordinate (5) at (0.5,1.5);
	
        \foreach \c in {3,4} {
	  \filldraw[black] (\c) circle (2pt);
    	}
        
    	\draw[rounded corners] (3) -- (1) --node[below] {$\frac{16}{33}$} (2) -- (4);
        \draw[rounded corners] (4) --node[below] {$\frac{4}{33}$} (3);
    	\draw[rounded corners] (3) --node[above left] {$\curv_e = \frac{13}{33}$} (5) -- (4);
		\begin{scope}[color=blue!30,opacity=0.5,rounded corners,line width=3pt,line cap=round,shorten <=4pt, shorten >=4pt]
			\draw (3) -- (1) -- (2) -- (4);
			\draw (3) -- (4);
			\draw (3) -- (5) -- (4);
		\end{scope}
        \end{tikzpicture}
    \end{minipage}
    \caption{Graph with edge resistances (left) and curvatures (right).}
    \label{fig:graph-subdiv}
    \end{figure}

    The house graph, in Figure~\ref{fig:house}, is a repeated edge subdivision of $(G, \ell)$.
    It is straightforward to check that the curvatures on $(G, \ell)$ are equal to sums of the corresponding curvatures in Figure~\ref{fig:house}.
\end{eg}

\section{Ricci flow}

In this section, we define and study a flow on edge-weighted graphs that uses the edge curvatures from the previous section.
This flow is inspired by Ricci flow~\cite{hamilton} in geometry: starting with a smooth manifold, we may try making it ``rounder'' by shrinking the parts with positive curvature, and expanding the parts with negative curvature.

\begin{dfn}
    The {\em Ricci(--Foster) flow} on a weighted graph $(G, \ellzero)$ is the system of differential equations on $\{\ell_e(t) : e \in E(G)\}$ defined by
    \begin{equation}\label{eq:ricci-flow}
        \tderiv{} \ell_{e}(t) = - \curv_{e}(t) 
        \qquad\text{and}\qquad
        \ell_e(0) = \ellzero_{e}
    \end{equation}
    for each edge $e \in E(G)$,
    where $\curv_{e}(t)$ depends on $\{\ell_f(t) : f \in E\}$ via \eqref{eq:curvature}.
\end{dfn}

From this definition, the intuition that we see from Proposition~\ref{prop:curvature-partials} is that applying the Ricci flow will have the effect of making the graph ``rounder," i.e. evening out extreme curvatures.  If an edge has very high positive curvature, the Ricci flow will decrease the length of that edge and thus decrease its curvature, and will likewise increase the length of a very negatively curved edge.

Earlier in Proposition~\ref{prop:curvature-bounds}, we showed that the curvature always lies in the range $-1 \leq \curv_e \leq 1$.
Thus under Ricci flow,
\[
    \ell_e(0) - t \leq \ell_e(t) \leq \ell_e(0) + t
\]
for all $t$ where the flow is defined.

\begin{prop}\label{prop:total-length-flow}
    Under Ricci flow, the total length of a graph decreases at rate one.
\end{prop}
(We use ``length'' and ``resistance'' interchangeably for $\ell_e$.)
\begin{proof}
    From Proposition~\ref{prop:curvature-sum}, we know that the edge curvatures $\curv_e$ sum to one.
\end{proof}

We record the following result for later use.
\begin{lem}\label{lem:curvature-flow}
    Under Ricci flow,
    $\displaystyle
        \tderiv{} \curv_e = - \frac{\omega_e}{\ell_e^2} \curv_e - \frac{1}{\ell_e} \tderiv{\omega_e}
    $.
\end{lem}
\begin{proof}
    Using the definition of curvature, we have
    \[
        \tderiv{} \curv_e = - \tderiv{} \left( \frac{\omega_e}{\ell_e}\right)
        = - \frac{1}{\ell_e} \tderiv{\omega_e} + \frac{\omega_e}{\ell_e^2} \tderiv{\ell_e}
        = - \frac{1}{\ell_e} \tderiv{\omega_e} - \frac{\omega_e}{\ell_e^2} \curv_e. \qedhere
    \]
\end{proof}

\subsection{Existence of Ricci flow}

In this section, we prove that the Ricci flow~\eqref{eq:ricci-flow} exists and is unique for an arbitrary initial graph $(G, \ellzero)$.
We apply the following theorem on ordinary differential equations.

\begin{thm}[Picard--Lindel\"{o}f]
Suppose $D \subset \RR \times \RR^n$ is a closed rectangle with $(t_0, y_0)$ in its interior, 
and suppose $f : D \to \RR^n$ is continuous in $t$ and Lipschitz continuous in $y$.
Then, there exists some $\epsilon > 0$ such that the initial value problem
\[
    \tderiv{} y(t)= f(t, y(t)), \qquad y(t_0) = y_0
\]
has a unique solution $y(t)$ on the interval $[t_0 - \epsilon, t_0 + \epsilon]$.
\end{thm}
For a proof, see Teschl~\cite[Theorem 2.2]{teschl}.

\begin{proof}[Proof of Theorem~\ref{thm:flow-existence}]
    For each edge $e$, the curvature $\curv_e(t, \ell(t))$ is a rational function of the edge lengths $\{\ell_f: f \in E\}$ and time $t$.
    The denominator is nonzero when all edge lengths are positive.
    Thus for an initial choice of positive edge lengths $\ellzero$, there is a closed rectangle $D_0 \subset \RR^{E(G)}$ containing $\ellzero$ such that the denominator of $\curv_e$ does not vanish on $D_0$ for any $e$.
    On $D = \{t = 0\} \times D_0 \subset \RR \times \RR^{E(G)}$, the curvature $\curv_e(t, \ell)$ is differentiable with a continuous derivative.
    Since $D$ is compact, this implies that $\curv_e(t, \ell)$ is Lipschitz on $D$.
\end{proof}

\begin{rmk}[Flow on metric graphs]\label{rmk:flow-metric}
    In Theorem~\ref{thm:curv-subdiv}, we showed that the edge curvature is compatible with edge refinement.
    As a consequence, Ricci flow defines a flow on the space of metric graphs.
\end{rmk}

\subsection{Examples}

\begin{eg}[Trees]
    Consider the tree shown in Figure~\ref{fig:tree-curv}, which has unit edge resistances.
    Under Ricci flow, the lengths of the edges will change as indicated in Figure~\ref{fig:tree-flow}.
    The flow is defined on the interval $t \in [0, T)$ for $T = 2$.
    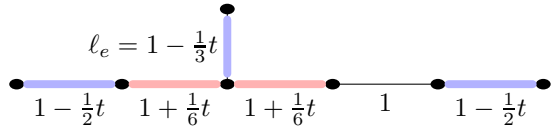
\begin{figure}[ht]
    \centering
    \begin{tikzpicture}[xscale=1.4]
        \coordinate (A) at (0,0);
        \coordinate (B) at (1,0);
        \coordinate (C) at (2,0);
        \coordinate (D) at (3,0);
        \coordinate (E) at (4,0);
        \coordinate (F) at (5,0);
        \coordinate (G) at (2,1);
        
        \draw (A) --node[below] {$1 - \frac12 t$} (B) --node[below] {$1 + \frac16 t$} (C) 
            --node[below] {$1 + \frac16 t$} (D) --node[below] {$1$} (E) --node[below] {$1 - \frac12 t$} (F);
        \draw (C) --node[left] {$\ell_e = 1 - \frac13 t$} (G);
        \foreach \p in {A, B, C, D, E, F, G} {
            \fill (\p) circle (2pt);
        }
		\begin{scope}[color=blue!30,opacity=0.5,line width=3pt,line cap=round,shorten <=4pt, shorten >=4pt]
			\draw (A) -- (B);
			\draw (C) -- (G);
			\draw (E) -- (F);
		\end{scope}
		\begin{scope}[color=red!30,opacity=0.5,line width=3pt,line cap=round,shorten <=4pt, shorten >=4pt]
			\draw (B) -- (C);
			\draw (C) -- (D);
		\end{scope}
    \end{tikzpicture}
    \caption{A tree under Ricci--Foster flow.}
    \label{fig:tree-flow}
    \end{figure}

    In general, on any tree $G$, the curvature $\curv_e$ depends only on the vertex degrees and not on the lengths $\ell_e$.
    Thus under Ricci flow, 
    \begin{itemize}
        \item edges incident to a leaf vertex (``terminal'' edges) will contract,
        \item non-terminal edges next to a branching vertex (degree $\geq 3$) will expand,
        \item all other edges will remain the same length.
    \end{itemize}
\end{eg}

\begin{eg}[Cycles]
    If $(G, \ell)$ is a cycle on $n$ vertices and edges, with edge lengths $\left\{\ellzero_1, \ldots, \ellzero_n\right\}$, then under Ricci flow each edge contracts at a rate proportional to its starting length.
    By Proposition~\ref{prop:curv-rescaling}, this implies that the curvatures remain constant over time, $\curv_e(t) = \curv_e^{(0)}$.
    Hence
    \[
        \ell_i(t) \;=\; \ellzero_i - \frac{\ellzero_i}{\ellzero_1 + \cdots + \ellzero_n} t
        \;=\; \ellzero_i \left( 1 - \lambda\, t\right) \qquad\text{where }\lambda = \frac{1}{\sum_j \ellzero_j}.
    \]
    This flow is defined on $t \in [0, T)$ for $T = \sum_i \ellzero_{i}$.

    In this case, the graph at time $t$ differs from the initial cycle by a global rescaling.
    A Riemannian manifold that has this property under (classical) Ricci flow is called an {\em Einstein manifold};
    see Section~\ref{sec:further} for more discussion.
\end{eg}

\begin{eg}[Barbell graph]
    Consider applying Ricci flow to the barbell graph in Figure~\ref{fig:barbell}, which has unit edge lengths.
    In light of Theorem~\ref{thm:curv-subdiv} and Remark~\ref{rmk:flow-metric}, we may replace the graph with its minimal metric graph representative, which has two vertices and three edges, with edge lengths as shown in Figure~\ref{fig:barbell-flow}.
    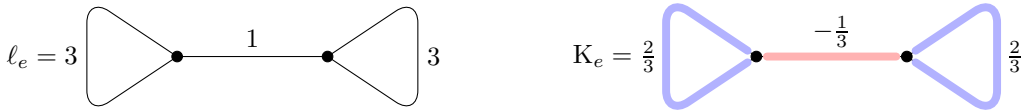
\begin{figure}[ht]
    \centering
    \raisebox{-0.5\height}{\begin{tikzpicture}
    	\coordinate (1) at (-1.2,0.8);
	  \coordinate (2) at (-1.2,-0.8);
    	\coordinate (3) at (0,0);
	    \coordinate (4) at (2,0);
    	\coordinate (5) at (3.2,-0.8);
        \coordinate (6) at (3.2,0.8);
	
	  \foreach \c in {3,4} {
	   	\filldraw[black] (\c) circle (2pt);
    	}
        \begin{scope}[rounded corners=10]
            \draw (3) -- (1) --node[left] {$\ell_e = 3$} (2) -- (3);
        	\draw (3) --node[above] {$1$} (4);
            \draw (4) -- (5) --node[right] {$3$} (6) -- (4);
            \end{scope}
    \end{tikzpicture}}
    \qquad\qquad
    \raisebox{-0.5\height}{\begin{tikzpicture}
    	\coordinate (1) at (-1.2,0.8);
	  \coordinate (2) at (-1.2,-0.8);
    	\coordinate (3) at (0,0);
	    \coordinate (4) at (2,0);
    	\coordinate (5) at (3.2,-0.8);
        \coordinate (6) at (3.2,0.8);
	
	  \foreach \c in {3,4} {
	   	\filldraw[black] (\c) circle (2pt);
    	}
        \begin{scope}[rounded corners=10]
            \draw (3) -- (1) --node[left] {$\curv_e = \frac{2}{3}$} (2) -- (3);
        	\draw (3) --node[above] {$-\frac{1}{3}$} (4);
            \draw (4) -- (5) --node[right] {$\frac{2}{3}$} (6) -- (4);
            \end{scope}
		\begin{scope}[opacity=0.5,rounded corners=10,line width=3pt,line cap=round,shorten <=4pt, shorten >=4pt]
			\draw[color=blue!30] (3) -- (1) -- (2) -- (3);
			\draw[color=blue!30] (4) -- (5) -- (6) -- (4);
			\draw[color=red!30] (3) -- (4);
        \end{scope}
    \end{tikzpicture}}
    \caption{Barbell graph and its edge curvatures.}
    \label{fig:barbell-flow}
    \end{figure}

    The edges curvatures are also indicated in the figure.
    Note that the edge curvatures depend only on the vertex degrees, and not on the lengths.
    Under Ricci flow, each loop edge has $\ell_e(t) = 3 - \frac23 t$, while the bridge edge has $\ell_e(t) = 1 + \frac13 t$.
    As $t \to T = \frac92$, each loop collapses to a point.
\end{eg}

\subsection{Ricci flow with surgery}

When applying Ricci flow to a graph $(G, \ell)$, we observed that the total length of the graph decreases over time (Proposition~\ref{prop:total-length-flow}). 
When the length of an edge reaches zero, Ricci flow is no longer defined on the original graph because we do not allow negative edge lengths.
However, we can consider continuing Ricci flow on the graph obtained by contracting the edge whose length vanishes.
This is reminiscent of the situation of Ricci flow on three-dimensional manifolds, where it is sometimes necessary to perform ``surgery'' near singularities~\cite{perelman-2}.

If we continue Ricci flow in this manner, contracting edges whenever their length approaches zero, then the flow starting from $(G, \ellzero)$ will continue until time $T = \sum_{e} \ellzero_e$, at which time it collapses to a single point.
This is a consequence of Proposition~\ref{prop:curvature-sum}.

\begin{eg}
    Consider the tree from Figure~\ref{fig:tree-curv}, which is shown at the top of Figure~\ref{fig:tree-flow-surgery}.
    The figure illustrates the process of Ricci flow starting from this tree, contracting an edge when its length becomes zero.
    The time axis is shown in the middle, pointing downward.
    On the initial graph, Ricci flow exists for time $0 \leq t < 2$, and at time $t = 2$ two terminal edges collapse.
    Additional edge collapses happen at times $t = 3$, $4$, $5$, and $6$.
    Each new graph after an edge collapse is shown along the left-hand side of the time axis.
    On the right-hand side, we show how edge lengths change under Ricci flow, and the sign of the curvature by color.
    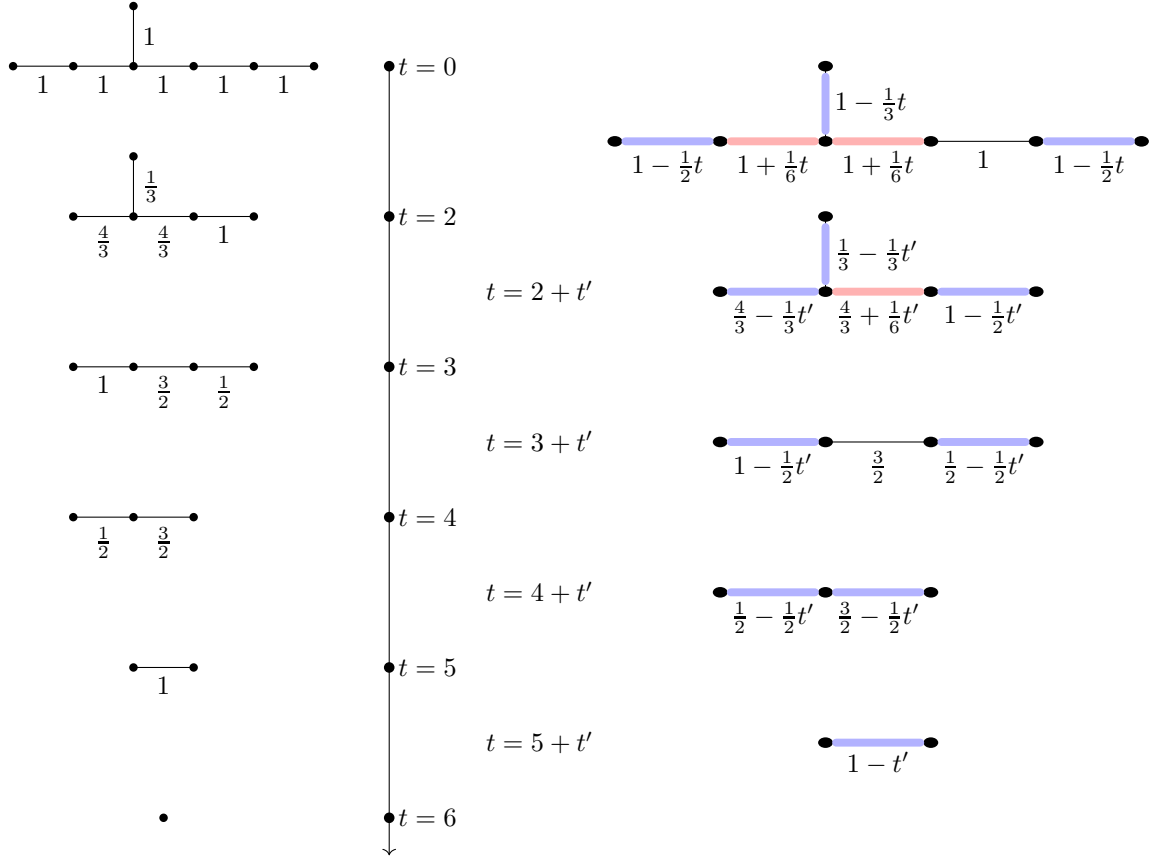
\begin{figure}[ht]
    \centering
    \begin{tikzpicture}[yscale=1.0]
        \draw (0,0) node[right] {$t = 0$} 
            -- (0,-2) node[right] {$t = 2$} 
            -- (0,-4) node[right] {$t = 3$}
            -- (0,-6) node[right] {$t = 4$}
            -- (0,-8) node[right] {$t = 5$}
            -- (0,-10) node[right] {$t = 6$};
        \draw[->] (0,-10) -- +(0,-0.5);
        \foreach \t in {0,-2,-4,-6,-8,-10} {
            \fill (0,\t) circle (2pt);
        }

    	\begin{scope}[shift={(-5,0)}, scale=0.8]
            \coordinate (A) at (0,0);
            \coordinate (B) at (1,0);
            \coordinate (C) at (2,0);
            \coordinate (D) at (3,0);
            \coordinate (E) at (4,0);
            \coordinate (F) at (5,0);
            \coordinate (G) at (2,1);
            
            \draw (A) --node[below] {$1$} (B) --node[below] {$1$} (C) 
                --node[below] {$1$} (D) --node[below] {$1$} (E) --node[below] {$1$} (F);
            \draw (C) --node[right] {$1$} (G);
            \foreach \p in {A, B, C, D, E, F, G} {
                \fill (\p) circle (2pt);
            }
        \end{scope}
        \begin{scope}[shift={(-5,-2)}, scale=0.8]
            \coordinate (B) at (1,0);
            \coordinate (C) at (2,0);
            \coordinate (D) at (3,0);
            \coordinate (E) at (4,0);
            \coordinate (G) at (2,1);
            
            \draw (B) --node[below] {$\frac43$} (C) 
                --node[below] {$\frac43$} (D) --node[below] {$1$} (E);
            \draw (C) --node[right] {$\frac13$} (G);
            \foreach \p in {B, C, D, E, G} {
                \fill (\p) circle (2pt);
            }
        \end{scope}
        \begin{scope}[shift={(-5,-4)}, scale=0.8]
            \coordinate (B) at (1,0);
            \coordinate (C) at (2,0);
            \coordinate (D) at (3,0);
            \coordinate (E) at (4,0);
            
            \draw (B) --node[below] {$1$} (C) 
                --node[below] {$\frac32$} (D) --node[below] {$\frac12$} (E);
            \foreach \p in {B, C, D, E} {
                \fill (\p) circle (2pt);
            }
        \end{scope}
        \begin{scope}[shift={(-5,-6)}, scale=0.8]
            \coordinate (B) at (1,0);
            \coordinate (C) at (2,0);
            \coordinate (D) at (3,0);
            
            \draw (B) --node[below] {$\frac12$} (C) 
                --node[below] {$\frac32$} (D);
            \foreach \p in {B, C, D} {
                \fill (\p) circle (2pt);
            }
        \end{scope}
        \begin{scope}[shift={(-5,-8)}, scale=0.8]
            \coordinate (C) at (2,0);
            \coordinate (D) at (3,0);
            
            \draw (C) --node[below] {$1$} (D);
            \foreach \p in {C, D} {
                \fill (\p) circle (2pt);
            }
        \end{scope}
        \begin{scope}[shift={(-5,-10)}, scale=0.8]
            \coordinate (C) at (2.5,0);
            \fill (C) circle (2pt);
        \end{scope}

        \begin{scope}[shift={(3,-1)}, xscale=1.4]
            \coordinate (A) at (0,0);
            \coordinate (B) at (1,0);
            \coordinate (C) at (2,0);
            \coordinate (D) at (3,0);
            \coordinate (E) at (4,0);
            \coordinate (F) at (5,0);
            \coordinate (G) at (2,1);
            
            \draw (A) --node[below] {$1 - \frac12 t$} (B) --node[below] {$1 + \frac16 t$} (C) 
                --node[below] {$1 + \frac16 t$} (D) --node[below] {$1$} (E) --node[below] {$1 - \frac12 t$} (F);
            \draw (C) --node[right] {$1 - \frac13 t$} (G);
            \foreach \p in {A, B, C, D, E, F, G} {
                \fill (\p) circle (2pt);
            }
    		\begin{scope}[color=blue!30,opacity=0.5,line width=3pt,line cap=round,shorten <=4pt, shorten >=4pt]
    			\draw (A) -- (B);
    			\draw (C) -- (G);
    			\draw (E) -- (F);
    		\end{scope}
    		\begin{scope}[color=red!30,opacity=0.5,line width=3pt,line cap=round,shorten <=4pt, shorten >=4pt]
    			\draw (B) -- (C);
    			\draw (C) -- (D);
    		\end{scope}
        \end{scope}
        \begin{scope}[shift={(2,-3)}]
            \node at (0,0) {$t = 2 + t'$};
        \end{scope}
        \begin{scope}[shift={(3,-3)}, xscale=1.4]
            \coordinate (B) at (1,0);
            \coordinate (C) at (2,0);
            \coordinate (D) at (3,0);
            \coordinate (E) at (4,0);
            \coordinate (G) at (2,1);
            
            \draw (B) --node[below] {$\frac43 - \frac13 t'$} (C) 
                --node[below] {$\frac43 + \frac16 t'$} (D) --node[below] {$1 - \frac12 t'$} (E);
            \draw (C) --node[right] {$\frac13 - \frac13 t'$} (G);
            \foreach \p in {B, C, D, E, G} {
                \fill (\p) circle (2pt);
            }
    		\begin{scope}[color=blue!30,opacity=0.5,line width=3pt,line cap=round,shorten <=4pt, shorten >=4pt]
    			\draw (B) -- (C);
                \draw (C) -- (G);
    			\draw (D) -- (E);
    		\end{scope}
    		\begin{scope}[color=red!30,opacity=0.5,line width=3pt,line cap=round,shorten <=4pt, shorten >=4pt]
    			\draw (C) -- (D);
    		\end{scope}
        \end{scope}
        \begin{scope}[shift={(2,-5)}]
            \node at (0,0) {$t = 3 + t'$};
        \end{scope}
        \begin{scope}[shift={(3,-5)}, xscale=1.4]
            \coordinate (B) at (1,0);
            \coordinate (C) at (2,0);
            \coordinate (D) at (3,0);
            \coordinate (E) at (4,0);
            
            \draw (B) --node[below] {$1 - \frac12 t'$} (C) 
                --node[below] {$\frac32$} (D) --node[below] {$\frac12 - \frac12 t'$} (E);
            \foreach \p in {B, C, D, E} {
                \fill (\p) circle (2pt);
            }
    		\begin{scope}[color=blue!30,opacity=0.5,line width=3pt,line cap=round,shorten <=4pt, shorten >=4pt]
    			\draw (B) -- (C);
    			\draw (D) -- (E);
    		\end{scope}
        \end{scope}
        \begin{scope}[shift={(2,-7)}]
            \node at (0,0) {$t = 4 + t'$};
        \end{scope}
        \begin{scope}[shift={(3,-7)}, xscale=1.4]
            \coordinate (B) at (1,0);
            \coordinate (C) at (2,0);
            \coordinate (D) at (3,0);
            
            \draw (B) --node[below] {$\frac12 - \frac12 t'$} (C) 
                --node[below] {$\frac32 - \frac12 t'$} (D);
            \foreach \p in {B, C, D} {
                \fill (\p) circle (2pt);
            }
    		\begin{scope}[color=blue!30,opacity=0.5,line width=3pt,line cap=round,shorten <=4pt, shorten >=4pt]
    			\draw (B) -- (C);
    			\draw (C) -- (D);
    		\end{scope}
        \end{scope}
        \begin{scope}[shift={(2,-9)}]
            \node at (0,0) {$t = 5 + t'$};
        \end{scope}
        \begin{scope}[shift={(3,-9)}, xscale=1.4]
            \coordinate (C) at (2,0);
            \coordinate (D) at (3,0);
            
            \draw (C) --node[below] {$1 - t'$} (D);
            \foreach \p in {C, D} {
                \fill (\p) circle (2pt);
            }
    		\begin{scope}[color=blue!30,opacity=0.5,line width=3pt,line cap=round,shorten <=4pt, shorten >=4pt]
    			\draw (C) -- (D);
    		\end{scope}
        \end{scope}
    \end{tikzpicture}
    \\
    \begin{tikzpicture}[xscale=1.4]
    \end{tikzpicture}
    \caption{A tree under Ricci flow with ``surgery.'' Edge contractions occur at the indicated times, and the tree collapses to a point at time $t = 6$. Edges are not drawn to scale.}
    \label{fig:tree-flow-surgery}
    \end{figure}
\end{eg}

\section{Nonnegatively curved graphs}

Recall that an edge-weighted graph $(G, \ell)$ is {\em nonnegatively curved} if every edge of $G$ has nonnegative Ricci--Foster curvature.
In this section, we show that Ricci flow preserves nonnegatively curved graphs.

\subsection{Some lemmas}

\begin{lem}\label{lem:res-homogeneous-pde}
    For any edge $e$, we have
    $\displaystyle
        \sum_{f \in E} \ell_f \frac{\partial\,\omega_e}{\partial\,\ell_f} = \omega_e
    $.
\end{lem}
\begin{proof}
    The effective resistance $\omega_e = \omega_e(\ell_1,\ldots,\ell_m)$, as a function of the edge lengths $\{\ell_f: f \in E\} = \{\ell_1, \ldots, \ell_m\}$ is homogeneous of degree one.
    In other words, for any positive real $\lambda$ we have
    \[
        \omega_e(\lambda\ell_1, \ldots, \lambda\ell_m) = \lambda\,\omega_e(\ell_1,\ldots,\ell_m).
    \]
    The claim follows from applying Euler's homogeneous function theorem.
\end{proof}


The following lemma concerns how the effective resistance $\omega_e$ evolves under Ricci flow.
Note that the conclusion in part (b) is a strengthening of part (a), since in (a) all terms $\omega_e$, $\ell_e$, and $\curv_e$ are nonnegative.
Also, note that part (b) does not make any assumptions on the sign of the curvature.
\begin{lem}\label{lem:res-flow-bound}
    Under Ricci flow, the following hold.
    \begin{enumerate}[label=(\alph*)]
        \item If $(G, \ell)$ is a nonnegatively curved, then $\displaystyle \tderiv{\omega_e}\Big|_{t = 0}$ is nonpositive for any edge $e$.

        \item If $e$ is chosen such that $\curv_e / \ell_e$ is minimal at $t = 0$, then $\displaystyle \tderiv{\omega_e}\Big|_{t = 0} + \omega_e \frac{\curv_e}{\ell_e}\Big|_{t = 0} \leq 0$.
    \end{enumerate}
\end{lem}
\begin{proof}
    If all curvatures are nonnegative, then under Ricci flow~\eqref{eq:ricci-flow} every edge resistance $\ell_e$ can only decrease over time.
    Rayleigh's monotonicity law (Proposition~\ref{prop:rayleigh}) states that this can only decrease the effective resistance between any vertices in the graph.
    This proves part (a).

    For part (b), we suppose that $e$ is chosen such that $\frac{\curv_e}{\ell_e} \leq \frac{\curv_f}{\ell_f}$ for all edges $f$, at time $t = 0$.
    Then
    \begin{align*}
        - \tderiv{\omega_e} 
        = - \sum_{f \in E} \partialderiv{\omega_e}{\ell_f} \tderiv{\ell_f} 
        &= \sum_{f \in E} \partialderiv{\omega_e}{\ell_f} \curv_f && \text{by Ricci flow \eqref{eq:ricci-flow}}\\
        &= \sum_{f \in E} \ell_f \partialderiv{\omega_e}{\ell_f} \left(\frac{\curv_f}{\ell_f}\right) \\
        &\geq \sum_{f \in E} \ell_f \partialderiv{\omega_e}{\ell_f} \left(\frac{\curv_e}{\ell_e}\right) && \text{since $\ell_f \partialderiv{\omega_e}{\ell_f} \geq 0$.}
    \end{align*}
    The last inequality applies Rayleigh's monotonicity law, Proposition~\ref{prop:rayleigh}.
    By Lemma~\ref{lem:res-homogeneous-pde}, we have
    $\displaystyle \sum_{f \in E} \ell_f \frac{\partial\,\omega_e}{\partial\,\ell_f} = \omega_e$ 
    so we conclude that $\displaystyle -\tderiv{\omega_e} \geq \omega_e \left(\frac{\curv_e}{\ell_e}\right)$ at $t = 0$, as desired.
\end{proof}

\subsection{Main proof}
\begin{proof}[Proof of Theorem~\ref{thm:positively-curved}]
    Suppose $(G, \ellzero)$ is a nonnegatively curved finite graph, i.e., $\curv_e(\ellzero) \geq 0$ for every edge $e$ in $E(G)$.
    We claim that as the edge lengths evolve under Ricci flow, the edge curvatures remain nonnegative.
    To verify this, it suffices to check that the ratios $\curv_e(t) / \ell_e(t)$ remain nonnegative.
    We consider the collection of real numbers 
    \[\left\{ \frac{\curv_e(t)}{\ell_e(t)} : e \in E(G)\right\}\]
    as the time $t$ varies.
    We claim that, as $t$ increases, the value $\min_{e \in E(G)} \left\{ \frac{\curv_e(t)}{\ell_e(t)}\right\}$ cannot decrease.

    For convenience, in the rest of the proof we will use $\curv_e$ and $\ell_e$ to denote $\curv_e(t)$ and $\ell_e(t)$, respectively.
    Suppose $e$ is chosen to minimize the value of $\curv_e / \ell_e$ (at some particular time $t$).
    We compute
    \begin{align*}
        \tderiv{} \left( \frac{\curv_e}{\ell_e} \right) 
        &= \frac{1}{\ell_e^2} \left(\ell_e \tderiv{\curv_e} - \curv_e \tderiv{\ell_e} \right) \\
        &= \frac{1}{\ell_e^2} \left(\ell_e \tderiv{\curv_e} + \curv_e^2 \right) &&\text{by Ricci flow \eqref{eq:ricci-flow}} \\
        &= \frac{1}{\ell_e^2} \left(-\omega_e \frac{\curv_e}{\ell_e} - \tderiv{\omega_e} + \curv_e^2\right) &&\text{by Lemma~\ref{lem:curvature-flow} for $\tderiv{\curv_e}$.} 
    \end{align*}
    Lemma~\ref{lem:res-flow-bound} states that $\displaystyle \omega_e \frac{\curv_e}{\ell_e} + \tderiv{\omega_e} \leq 0$.
    Thus
    \[
        \tderiv{} \left( \frac{\curv_e}{\ell_e} \right) = \frac{1}{\ell_e^2} \left(-\omega_e \frac{\curv_e}{\ell_e} - \tderiv{\omega_e} + \curv_e^2\right)
        \geq \frac{1}{\ell_e^2} \curv_e^2 \geq 0.
    \]
    This shows that $\min_{e \in E(G)} \left\{  \frac{\curv_e(t)}{\ell_e(t)}\right\}$ is nondecreasing over time under Ricci flow, so all curvatures $\curv_e$ stay nonnegative.
\end{proof}

\begin{rmk}
    The proof of Theorem~\ref{thm:positively-curved} shows that for any initial graph $(G, \ellzero$), without assumptions on curvature, the value of
    $\displaystyle
        \min \left\{ \frac{\curv_e(t)}{\ell_e(t)} : e \in E(G) \right\}
    $
    is nondecreasing over time.
    Thus in particular, positively-curved graphs remain positively-curved. On some domain $[0,T')\subseteq[0,T)$ this is a consequence of continuity, but this brings the result to the full domain on which Ricci flow exists.
\end{rmk}

\section{Further questions}\label{sec:further}

Here we discuss some remaining open questions which would be interesting for further investigation.  

In geometry, we say a manifold $(M, g)$ has an {\em Einstein metric} if the Ricci curvature is a multiple of the metric, $\mathrm{Ric} = \lambda \, g$.
For such a manifold, Ricci flow causes the metric to shrink or expand ``homothetically,'' meaning that the metric $g(t)$ at any time $t$ only differs from $g$ by a constant factor.~\cite[Chapter 1.2.1]{topping}.
We can make the analogous definition for Ricci--Foster curvature.
\begin{dfn}
    An edge-weighted graph $(G, \ell)$ is an {\em Einstein network} if there is a constant $\lambda \in \RR$ such that
    \[
        \curv_e = \lambda \,\ell_e \qquad\text{for all } e \in E(G).
    \]
\end{dfn}
By Proposition~\ref{prop:curv-rescaling}, for an Einstein network the curvature $\curv_e$ remains constant under Ricci flow, and the graphs $(G, \ell(t))$ shrink or expand homothetically.
(By Proposition~\ref{prop:curvature-sum}, a finite graph can only shrink homothetically.)
Earlier, we saw the $n$-cycle is an example of an Einstein network.
By symmetry reasons, if $(G, \ell)$ is any edge-transitive graph (with length-preserving symmetries), then it is an Einstein network.
Are there any other examples?

\begin{prob}
    Which graphs $G$ admit a choice of edge lengths $\{\ell_e : e \in E(G)\}$ such that $(G, \ell)$ is an Einstein network?
\end{prob}



For manifolds, Perelman~\cite{perelman} proved that Ricci flow does not have any periodic orbits on the set of manifolds up to diffeomorphism, other than the constant orbits (e.g. Einstein manifolds).
\begin{prob}\label{prob:periodic}
    Does Ricci flow have any periodic orbits, on the space of graphs $(G, \ell)$ up to scaling?  
\end{prob}
We conjecture that the answer to Problem~\ref{prob:periodic} is ``no.''
Finally, we mention two additional questions.
\begin{prob}
    Can any improvements be given regarding the interval on which we can show the existence of the Ricci flow?
\end{prob}

\begin{prob}
    A discrete Ricci flow based on Olliver curvature has been applied in community detection~\cite{NLLG}. What applications, if any, does Ricci--Foster flow defined in \eqref{eq:ricci-flow-intro} admit? 

\end{prob}

\bibliographystyle{alpha}
\bibliography{ricci-flow-ref}

\end{document}